\documentclass[12pt]{amsart}
\usepackage{amsmath,amssymb}
\usepackage{amsfonts}
\usepackage{amsthm}
\usepackage{latexsym}
\usepackage{graphicx}

\def\lf{\left}
\def\ri{\right}

\def\dbar{\bar\partial}

\def\ii{\sqrt{-1}}

\def\vv<#1>{\langle#1\rangle}

\def\bi{{\bar i}}
\def\bj{{\bar j}}
\def\bk{{\bar k}}

\def\bl{{\bar l}}
\def\bla{{\bar \lambda}}
\def\bmu{{\bar\mu}}
\def\lam{{\lambda}}

\def\a{{\alpha}}
\def\ba{{\bar\alpha}}

\def\XXint#1#2{\setbox0=\hbox{$#1{#2}{\int}$}{#2}\kern-.5\wd0 }

\def\XXint#1#2#3{{\setbox0=\hbox{$#1{#2#3}{\int}$}
     \vcenter{\hbox{$#2#3$}}\kern-.5\wd0}}

\def\vv<#1>{\langle#1\rangle}

\def\na{{\nabla}}
\def\cs#1{\left(#1\right)}
\newtheorem{thm}{Theorem}[section]

\newtheorem{lem}{Lemma}[section]
\newtheorem{prop}{Proposition}[section]
\newtheorem{cor}{Corollary}[section]
\theoremstyle{definition}
\newtheorem{defn}{Definition}[section]
\theoremstyle{remark}

\newtheorem{rem}{Remark}[section]

\numberwithin{equation}{section}
\begin{document}
\title{Hessian comparison and Eigenvalue of almost Hermitian manifolds}
\author{Chengjie Yu$^1$ }
\address{Department of Mathematics, Shantou University, Shantou, Guangdong, 515063, China}
\email{cjyu@stu.edu.cn}
\thanks{$^1$Research partially supported by the National Natural Science Foundation of
China (11001161),(10901072) and (11101106).}

\renewcommand{\subjclassname}{%
  \textup{2000} Mathematics Subject Classification}
\date{}
\begin{abstract}
In this paper, by using the Bochner technique on almost Hermitian manifolds, we obtain a complex Hessian comparison for almost Hermitian manifolds generalizing the Laplacian comparison for almost Hermitian manifolds by Tossati, and reprove a diameter estimate for almost Hermitian manifolds by Gray. Moreover, we obtain a sharp eigenvalue estimate on quasi K\"ahler manifolds and a sharp Hessian comparison on nearly K\"ahler manifolds.
\end{abstract}
\maketitle\markboth{Chengjie Yu}{Hessian comparison and eigenvalue comparison}
\section{Introduction}
A triple $(M,J,g)$ is called an almost Hermitian manifold if $J$ is an almost complex structure and $g$ is a $J$-invariant Riemannian metric. There are two connections, one is the Levi-Civita connection and the other one is the canonical connection, on almost Hermitian manifolds, that play important roles on the geometry of almost Hermitian manifolds. The canonical connection is an extension of the Chern connection \cite{Chern} on Hermitian manifolds. It was first introduced by Ehresmann-Libermann \cite{e}.

Geometers were used to use the Levi-Civita connection for the study of the geometry of almost Hermitian manifolds, see for example \cite{Goldberg,Gray1,Gray2,Gray3,AD}. However, later researches show that canonical connection is useful for the study of the geometry of almost Hermitian manifolds. For example, canonical connection is crucial for the study of the structure of nearly K\"ahler manifolds in \cite{Na1,Na2,B}. In \cite{twy}, Tossati, Weinkove and Yau used the canonical connection to solve the Calabi-Yau equation on almost K\"ahler manifolds. The problem Tossati-Weinkove-Yau considered is part of a program proposed by Donaldson \cite{Donaldson,D2} on sympletic topology. In \cite{vt}, Tossati obtained a Laplacian comparison result about the canonical connection  on almost Hermitian manifolds using the second variation of arc length and obtained a Schwartz lemma on almost Hermitian manifolds which is a generalization of the Schwartz lemma by Yau \cite{Yau}.

In this paper, by applying the same Bochner technique as in \cite{LW}, we obtain a Hessian comparison on almost Hermitian manifolds which generalises Tossati's Laplacian comparison \cite{vt}. More precisely, we obtain the following result.
\begin{thm}Let $(M,J,g)$ be a complete almost Hermitian manifold with holomorphic bisectional
curvature bounded from blow by $-K$ with $K\geq 0$, torsion bounded by $A_1$ and the (2,0) part of
the curvature tensor bounded by $A_2$. Then
\begin{equation}
\rho_{i\bj}\leq \left[\frac{1}{\rho}+\left((4\sqrt n+3)A_1^{2}+2A_2+K\right)^\frac{1}{2}\right]g_{i\bj}
\end{equation}
within the cut-locus of $o$.
\end{thm}
Moreover, with the same technique, we obtain the following sharp diameter estimate for almost Hermitian manifolds.
\begin{thm}
Let $(M,J,g)$ be a complete almost Hermitian manifold and the quasi holomorphic sectional
curvature is not less than $K>0$. Then $d(M)\leq \pi/\sqrt K$.
\end{thm}
In fact, the above diameter estimate was disguised with a seemingly different curvature assumption in \cite{Gray1}. However, one can show that the two curvature assumptions are the same by using the curvature identities derived in \cite{Yu3}. The same diameter estimate for Hermitian manifolds was also obtain in \cite{CY}.

 Furthermore, by using a similar technique as in Futaki \cite{F}, we have the following first eigenvalue estimate on almost Hermitian manifolds.
\begin{thm}
Let $(M,J,g)$ be a compact quasi K\"ahler manifold with the quasi Ricci curvature bounded
from below by a positive constant $K$. Then $\lam_1\geq 2K$, where $\lam_1$ is the first
eigenvalue of $(M,g)$.
\end{thm}

Finally, we obtain a sharp Hessian comparison on nearly K\"ahler manifolds which  generalizes some results in \cite{LW,TY2} on K\"ahler manifolds.
\begin{thm}
Let $(M,J,g)$ be a complete nearly K\"ahler manifold and $o$ be a fixed point in $M$. Let $B_o(R)$ be a geodesic ball within the cut-locus of $p$. Suppose that the quasi holomorphic bisectional curvature on $B_o(R)$ is not less than $K$ where $K$ is a constant. Then
\begin{equation}
 \rho_{\alpha\bar\beta}\leq\left\{\begin{array}{ll}
 \sqrt{K/2}\cot(\sqrt{K/2}\rho)(g_{\alpha\bar\beta}-2\rho_\alpha \rho_{\bar\beta})+\sqrt{2K}\cot(\sqrt{2K}\rho)\rho_\alpha \rho_{\bar\beta}&(K>0)\\
 \frac{1}{r}(g_{\alpha\bar\beta}-\rho_\alpha \rho_{\bar\beta})&(K=0)\\
 \sqrt{-K/2}\coth(\sqrt{-K/2}\rho)(g_{\alpha\bar\beta}-2\rho_\alpha \rho_{\bar\beta})+\sqrt{-2K}\coth(\sqrt{-2K}\rho)\rho_\alpha \rho_{\bar\beta}&(K<0)\\
 \end{array}\right.
 \end{equation}
 in $B_o(R)$ with equality holds all over $B_o(R)$ if and only if $B_o(R)$ is holomorphic and isometric equivalent to the geodesic ball with radius $R$ in the K\"ahler space form of constant holomorphic bisectional curvature $K$, where $\rho$ is the distance function to the fixed point $o$.
\end{thm}

\section{Hessian comparison and diameter estimate on almost Hermitian manifolds}
 we first recall some definitions and known results in almost Hermitian geometry. 
 
 \begin{defn}[\cite{k,k2,g}] Let $(M,J)$ be an almost complex manifold. A Riemannian metric $g$ on $M$ such that $g(JX,JY)=g(X,Y)$ for any two tangent vectors $X$ and $Y$ is called an almost Hermitian metric.  The triple $(M,J,g)$ is called an almost Hermitian manifold. The two form $\omega_g=g(JX,Y)$ is called the fundamental form of the almost Hermitian manifold. A connection $\nabla$ on an almost Hermitian manifold $(M,J,g)$ such that $\nabla g=0$ and $\nabla J=0$ is called an almost Hermitian connection.
\end{defn}

Note  that the torsion $\tau$ of the connection $\nabla$ is a vector-valued two form defined as
\begin{equation}
\tau(X,Y)=\nabla_XY-\nabla_YX-[X,Y].
\end{equation}
 An almost Hermitian connection is uniquely determined by its (1,1)-part. In particular, there is a unique almost Hermitian connection with vanishing (1,1)-part. Such a connection is called the canonical connection which is first introduced by Ehresman and Libermann \cite{e}.
\begin{defn}[\cite{k,k2}]The unique almost Hermitian connection $\nabla$ on an almost Hermitian manifold $(M,J,g)$ with vanishing $(1,1)$-part of the torsion is called the canonical connection of the almost Hermitian manifold.
\end{defn}

For sake of convenience, we adopt the following conventions in the remaining part of this paper:
\begin{enumerate}
\item Without further indications, the manifold is of real dimension $2n$;
\item $D$ denotes the Levi-Civita connection and $R^L$ denotes its curvature tensor and ''$,$'' means taking covariant derivatives with respect to $D$;
\item $\nabla$ denotes the canonical connection,$R$ denote the curvature tensor of $\nabla$ and ''$;$'' means taking covariant derivatives with respect to $\na$.
\item Without further indications, English letters such as $a,b,c$ etc  denote indices in $\{1,\bar1,2,\bar 2,\cdots,n,\bar n\}$;
\item Without further indications, $i,j,k$ etc denote indices in $\{1,2,\cdots,n\}$.
\item Without further indications, Greek letters such as $\lambda,\mu$ denote summation indices going through $\{1,2,\cdots,n\}$.
\end{enumerate}
Recall the definition of curvature operator:
\begin{equation}
R(X,Y)Z=\nabla_X\nabla_YZ-\nabla_Y\nabla_XZ-\nabla_{[X,Y]}Z.
\end{equation}
The curvature tensor is defined as
\begin{equation}
R(X,Y,Z,W)=\vv<R(Z,W)X,Y>.
\end{equation}
Fixed a unitary $(1,0)$-frame $(e_1,e_2,\cdots,e_n)$, since $\nabla J=0$, we have
\begin{equation}
R_{ijab}=R_{i\ ab}^{\ \bj}=0
\end{equation}
for all indices $i,j$ and $a,b$. Moreover, similarly as in the Riemannian case, we have the following symmetries of the curvature tensor:
\begin{equation}
R_{abcd}=-R_{bacd}=-R_{abdc}
\end{equation}
for all indices $a,b,c$ and $d$. Recall that $R'_{ab}=g^{\bmu\lam}R_{\lam\bmu ab}$ and  $R''_{i\bj}=g^{\bmu\lam}R_{i\bj\lam\bmu}$ are called the first and the second Ricci curvature of the almost Hermitian metric $g$ respectively.

The following first Bianchi identities for almost Hermitian manifolds are frequently used in the computations of the remaining part of this paper. One can find them in \cite{twy,k,Yu3}.
\begin{prop}\label{prop-first-bian}
Let $(M,J,g)$ be an almost Hermitian manifold. Fixed a unitary frame, we have
\begin{enumerate}
\item $R_{i\bj k\bar l}-R_{k\bj i\bar l}=\tau^j_{ik;\bar
l}-\tau^{\bar \lambda}_{ik}\tau^j_{\bar l\bar\lambda}$;
\item $R_{i\bj  k\bl }-R_{i \bl k\bj }=\tau^\bi_{\bj\bl;
k}-\tau^\bi_{k\lambda}\tau^{\lambda}_{\bj\bl}$;
\item $R_{i\bj k\bl}-R_{k\bl i\bj}=\tau_{ik;\bj}^l+\tau_{\bj\bl;k}^\bi-\tau_{k\lam}^\bi\tau_{\bj\bl}^\lam-\tau_{ik}^\bla\tau_{\bj\bla}^l$;
\item $R_{i\bj kl}=-\tau_{kl;\bj}^\bi+\tau_{\bj\bla}^\bi\tau_{kl}^\bla$.
\end{enumerate}
\end{prop}

The following general Ricci identity for commuting indices is also frequently used in the remaining part of this paper. One can also find it in \cite{FTY}.
\begin{lem}\label{lem-ricci-identity}
Let $M^n$ be a smooth manifold, and $E$ be a vector bundle on $M$. Let $D$ be a connection on $E$ and $\nabla$ be a connection on $M$ with torsion $\tau$. Then
$$D^2s(X,Y)-D^2s(Y,X)=-R(X,Y)s+D_{\tau(X,Y)}s$$
for any cross section $s$ of $E$, and tangent vector fields $X$ and $Y$.
\end{lem}

Directly by the Ricci identity above, we have
\begin{equation}
f_{i\bj}=f_{\bj i}
\end{equation}
and
\begin{equation}
f_{ij}-f_{ji}=\tau_{ij}^\lam f_\lam+\tau_{ij}^\bla f_\bla
\end{equation}
on almost Hermitian manifolds.

Moreover, recall the following comparisons of geometric quantities for the Levi-Civita connection and the canonical connection on almost Hermitian manifolds.

\begin{lem}[\cite{g,twy,FTY}]\label{lem-comp-connection}
Let $(M,J,g)$ be an almost Hermitian manifold. Then
$$\vv<D_YX-\nabla_Y X,Z>=\frac{1}{2}\cs{\vv<\tau(X,Y),Z>+\vv<\tau(Y,Z),X>-\vv<\tau(Z,X),Y>}.$$
\end{lem}
By using Lemma \ref{lem-comp-connection} directly, we have the following comparisons of the Hessian and divergence operators with respect to the Levi-Civita connection and the canonical connection.
\begin{lem}\label{lem-comp-hess}
On an almost Hermitian manifold, fixed a unitary frame,
\begin{equation}
f_{i\bj}-f_{,i\bj}=\frac{1}{2}(\tau_{i\lam}^jf_\bla +\tau_{\bj\bla}^\bi f_\lam)
\end{equation}
where ',' means taking covariant derivatives with respect to the Levi-Civita connection.
\end{lem}
\begin{lem}\label{lem-comp-hess-2}
On an almost Hermitian manifold, fixed a unitary frame,
\begin{equation}
f_{i j}-f_{,ij}=\frac{1}{2}(\tau_{ij}^\lam f_\lam+\tau_{ij}^\bla f_\bla+\tau_{i\lam}^\bj f_\bla +\tau_{j\lam}^\bi f_\bla)
\end{equation}
where ',' means taking covariant derivatives with respect to the Levi-Civita connection.
\end{lem}
\begin{lem}\label{lem-comp-laplace}
On an almost Hermitian manifold, fixed a unitary frame,
\begin{equation}
\Delta f-\Delta^Lf=\tau_{i\lam}^if_\bla+\tau_{\bi\bla}^\bla f_\lam
\end{equation}
where $\Delta^L$ is the Laplacian operator with respect to the Levi-Civita connection.
\end{lem}
\begin{lem}\label{lem-comp-div}
Let $X$ be a vector field on an almost Hermitian manifold $M$ and fixed a unitary frame. Then
\begin{equation}
\mbox{div} X-\mbox{div}_LX=X^i\tau_{j i}^j+X^\bi\tau_{\bj\bi}^\bj
\end{equation}
where $\mbox{div}X=X^i_{;i}+X^{\bi}_{;\bi}$ is the divergence of $X$ with respect to the canonical connection and $\mbox{div}_L X$ is the divergence of $X$ with respect to the Levi-Civita connection.
\end{lem}

The same as in Tosatti \cite{vt}, we make the following definition about the bound-ness of the curvatures of an almost Hermitian manifold.
\begin{defn}
Let $(M,J,g)$ be an almost Hermitian manifold. We say that the holomorphic bisectional curvature of $(M,J,g)$ is bounded from below by $K$ if
\begin{equation}
R(X,\bar X,Y,\bar Y)\geq K \|X\|^2\|Y\|^2
\end{equation}
for any $X,Y\in T^{1,0}M$. We say that the torsion of $(M,J,g)$ is bounded by $A_1$ if
\begin{equation}
\|\tau(X,Y)\|\leq A_1\|X\|\|Y\|
\end{equation}
for any $X,Y\in T^{1,0}M$. We say that the (2,0) part of the curvature tensor of $(M,J,g)$ is bounded by $A_2$ if
\begin{equation}
|R(\bar X,Y,Y,X)|\leq A_2\|X\|^2\|Y\|^2
\end{equation}
for any $X,Y\in T^{1,0}M$.
\end{defn}

Let $(M,J,g)$ be an almost Hermitian manifold. We denote its distance function to fixed point $o$ as $\rho$. Similarly as in Li-Wang \cite{LW}, we have the follows.
\begin{lem}\label{lem-evolv-1-order}
Fixed a unitary frame $(e_1,e_2,\cdots,e_n)$, we have
\begin{equation}
\rho_{ij}\rho_\bi+\rho_i\rho_{\bi j}=0\ \mbox{and}\ \rho_{i\bj}\rho_\bi+\rho_i\rho_{\bi \bj}=0.
\end{equation}
\end{lem}
\begin{proof}
Note that $\rho_i\rho_\bi=\frac{1}{2}$. Hence
\begin{equation}
0=(\rho_i\rho_\bi)_a=\rho_{ia}\rho_\bi+\rho_i\rho_{\bi a}.
\end{equation}
\end{proof}
\begin{lem}\label{lem-evolv-2-order}
Fixed a unitary frame $(e_1,e_2,\cdots,e_n)$, we have
\begin{equation}
\begin{split}
&\rho_{k\bl i}\rho_\bi+\rho_{k\bl \bi}\rho_i\\
=&-\rho_{i\bl}\rho_{\bi k}-\rho_\bi\tau_{ik}^\lambda\rho_{\lambda\bl}-\rho_{k\bla}\tau_{\bi\bl}^{\bla}\rho_i-\rho_{ik}\rho_{\bi\bl}-\rho_{\lam k}\tau_{\bi\bl}^\lambda\rho_i-\rho_\bi\tau_{ik}^\bla\rho_{\bla\bl}\\
&-(R_{\bl\lambda ik}+\tau_{ik}^\bmu\tau_{\bl\bmu}^\bla)\rho_\bla\rho_\bi-(R_{k\bla\bi\bl}+\tau_{k\mu}^\lam\tau_{\bi\bl}^\mu)\rho_\lambda\rho_i-(R_{i\bla k\bl}+\tau_{k\mu}^\bi\tau_{\bla\bl}^\mu+\tau_{ik}^\bmu\tau_{\bl\bmu}^\lam)\rho_\lambda\rho_\bi\\
\end{split}
\end{equation}
\end{lem}
\begin{proof}
Note that $\rho_i\rho_\bi=\frac{1}{2}$. Hence
\begin{equation}\label{eqn-2-order}
\begin{split}
0=&(\rho_i\rho_\bi)_{k\bl}\\
=&\rho_{ik}\rho_{\bi\bl}+\rho_{i\bl}\rho_{\bi k}+\rho_{ik\bl}\rho_\bi+\rho_{\bi k\bl}\rho_i\\
=&\rho_{ik}\rho_{\bi\bl}+\rho_{i\bl}\rho_{\bi k}+\cs{\rho_{ki}+\tau_{ik}^\lam\rho_\lam+\tau_{ik}^\bla\rho_\bla}_\bl\rho_\bi+\rho_{k\bi\bl}\rho_i\\
=&\rho_{ik}\rho_{\bi\bl}+\rho_{i\bl}\rho_{\bi k}+\cs{\rho_{ki\bl}+\tau_{ik;\bl}^\lam\rho_\lam+\tau_{ik}^\lam\rho_{\lam\bl}+\tau_{ik;\bl}^\bla\rho_\bla+\tau_{ik}^\bla\rho_{\bla\bl}}\rho_\bi\\
&+\cs{\rho_{k\bl\bi}+R_{k\bla\bi\bl}\rho_\lambda+\tau_{\bi\bl}^\lambda\rho_{k\lambda}+\tau_{\bi\bl}^\bla\rho_{k\bla}}\rho_i\\
=&\rho_{ik}\rho_{\bi\bl}+\rho_{i\bl}\rho_{\bi k}+(\rho_{k\bl i}+R_{k\bla i\bl}\rho_\lambda+\tau_{ik;\bl}^\lambda\rho_\lambda+\tau_{ik;\bl}^\bla\rho_\bla+\tau_{ik}^\lambda\rho_{\lambda\bl}+\tau_{ik}^\bla\rho_{\bla\bl})\rho_\bi\\
&+(\rho_{k\bl\bi}+R_{k\bla\bi\bl}\rho_\lambda+\tau_{\bi\bl}^\lambda\rho_{k\lambda}+\tau_{\bi\bl}^\bla\rho_{k\bla})\rho_i\\
=&\rho_{ik}\rho_{\bi\bl}+\rho_{i\bl}\rho_{\bi k}+[\rho_{k\bl i}+R_{i\bla k\bl}\rho_\lambda-\tau_{\bl\bmu}^\lambda\tau_{ki}^\bmu\rho_\lambda+(R_{\bl\lambda ik}+\tau_{\bl\bmu}^{\bla}\tau_{ik}^\bmu)\rho_\bla+\tau_{ik}^\lambda\rho_{\lambda\bl}\\
&+\tau_{ik}^\bla(\rho_{\bl\bla}+\tau_{\bla\bl}^\mu\rho_\mu+\tau_{\bla\bl}^{\bmu}\rho_\bmu)]\rho_\bi+(\rho_{k\bl\bi}+R_{k\bla\bi\bl}\rho_\lambda+\tau_{\bi\bl}^\lambda\rho_{k\lambda}+\tau_{\bi\bl}^\bla\rho_{k\bla})\rho_i\\
=&(\rho_{k\bl i}\rho_\bi+\rho_{k\bl \bi}\rho_i)+\rho_{i\bl}\rho_{\bi k}+\rho_\bi\tau_{ik}^\lambda\rho_{\lambda\bl}+\rho_{k\bla}\tau_{\bi\bl}^{\bla}\rho_i+\rho_{ik}\rho_{\bi\bl}+\rho_{k\lambda}\tau_{\bi\bl}^\lambda\rho_i+\rho_\bi\tau_{ik}^\bla\rho_{\bl\bla}\\
&+R_{\bl\lambda ik}\rho_\bla\rho_\bi+R_{k\bla\bi\bl}\rho_\lambda\rho_i+R_{i\bla k\bl}\rho_\lambda\rho_\bi\\
=&(\rho_{k\bl i}\rho_\bi+\rho_{k\bl \bi}\rho_i)+\rho_{i\bl}\rho_{\bi k}+\rho_\bi\tau_{ik}^\lambda\rho_{\lambda\bl}+\rho_{k\bla}\tau_{\bi\bl}^{\bla}\rho_i+\rho_{ik}\rho_{\bi\bl}+(\rho_{\lam k}+\tau_{k\lam}^\mu\rho_\mu+\tau_{k\lam}^\bmu\rho_\bmu)\tau_{\bi\bl}^\lambda\rho_i\\
&+\rho_\bi\tau_{ik}^\bla(\rho_{\bla\bl}+\tau_{\bl\bla}^\mu\rho_\mu+\tau_{\bl\bla}^\bmu\rho_\bmu)+R_{\bl\lambda ik}\rho_\bla\rho_\bi+R_{k\bla\bi\bl}\rho_\lambda\rho_i+R_{i\bla k\bl}\rho_\lambda\rho_\bi\\
=&(\rho_{k\bl i}\rho_\bi+\rho_{k\bl \bi}\rho_i)+\rho_{i\bl}\rho_{\bi k}+\rho_\bi\tau_{ik}^\lambda\rho_{\lambda\bl}+\rho_{k\bla}\tau_{\bi\bl}^{\bla}\rho_i+\rho_{ik}\rho_{\bi\bl}+\rho_{\lam k}\tau_{\bi\bl}^\lambda\rho_i+\rho_\bi\tau_{ik}^\bla\rho_{\bla\bl}\\
&+(R_{\bl\lambda ik}+\tau_{ik}^\bmu\tau_{\bl\bmu}^\bla)\rho_\bla\rho_\bi+(R_{k\bla\bi\bl}+\tau_{k\mu}^\lam\tau_{\bi\bl}^\mu)\rho_\lambda\rho_i+(R_{i\bla k\bl}+\tau_{k\mu}^\bi\tau_{\bla\bl}^\mu+\tau_{ik}^\bmu\tau_{\bl\bmu}^\lam)\rho_\lambda\rho_\bi\\
\end{split}
\end{equation}
where we have used Proposition \ref{prop-first-bian} and Lemma \ref{lem-ricci-identity}. This completes the proof.
\end{proof}
\begin{thm}Let $(M,J,g)$ be a complete almost Hermitian manifold with holomorphic bisectional
curvature bounded from blow by $-K$ with $K\geq 0$, torsion bounded by $A_1$ and the (2,0) part of
the curvature tensor bounded by $A_2$. Then
\begin{equation}
\rho_{i\bj}\leq \left[\frac{1}{\rho}+\left((4\sqrt n+3)A_1^{2}+2A_2+K\right)^\frac{1}{2}\right]g_{i\bj}
\end{equation}
within the cut-locus of $o$.
\end{thm}
\begin{proof}
Let $\gamma$ be normal geodesic starting from $o$. Let $(e_1,e_2,\cdots,e_n)$
be a parallel unitary frame along $\gamma$. Let $X=\cs{\rho_{k\bl}}$, $A=\cs{\rho_\bla\tau_{\lam k}^l}$,
$B=\cs{\rho_{\bk\bl}}$, $C=\cs{\tau_{\bla\bl}^k\rho_\lam}$, $D=\cs{(R_{k\bla \bar\nu\bl}+\tau_{k\mu}^\lam\tau_{\bar\nu\bl}^\mu)\rho_\lam\rho_\nu}$
and $E=\cs{(R_{\mu\bla k\bl}+\tau_{k\nu}^\bmu\tau_{\bla\bl}^\nu+\tau_{\mu k}^{\bar\nu}\tau_{\bl\bar\nu}^\lam)\rho_\lam\rho_\bmu}$

Then, by Lemma \ref{lem-evolv-2-order},
we know that
\begin{equation}\label{eqn-ricati}
\begin{split}
&\frac{dX}{d\rho}+X^2+AX+XA^*\\
=&-B^*B-B^*C-C^*B-(D+D^*)-E\\
=&-(B^*+C^*)(B+C)+C^*C-(D+D^*)-E\\
\leq&C^*C-(D+D^*)-E.\\
\end{split}
\end{equation}
Moreover, for any column vector $u$, we have
\begin{equation}\label{eqn-1}
u^*C^*Cu=\|Cu\|^2\leq \frac{1}{2}A_1^2\|u\|^2,
\end{equation}
\begin{equation}
\left|\sum_{k,l,\lam,\nu=1}^n\bar u_kR_{k\bla\bar\nu l}\rho_\lam\rho_\nu u_l\right|\leq \frac{1}{2}A_2\|u\|^2,
\end{equation}
\begin{equation}\label{eqn-tau-square}
\begin{split}
&\left|\sum_{k,l,\mu,\nu,\lam=1}^n\bar u_k\tau_{k\mu}^\lam\rho_\lam\tau_{\bar\nu\bl}^\mu\rho_\nu u_l\ri|\\
\leq&\left(\sum_{\mu=1}^n\lf|\sum_{k,\lam=1}^n\bar u_k\tau_{k\mu}^\lam\rho_\lam\ri|^2\right)^\frac{1}{2}\left(\sum_{\mu=1}^n\lf|\sum_{\nu,l=1}^n\tau_{\bar\nu\bl}^\mu\rho_\nu u_l\ri|^2\right)^\frac{1}{2}\\
\leq&\left(\sum_{\lam,\mu=1}^n\lf|\sum_{k=1}^n\bar u_k\tau_{k\mu}^\lam\ri|^2\sum_{\lam=1}^n|\rho_\lam|^2\right)^\frac{1}{2}\times \frac{1}{\sqrt2}A_1\|u\|\\
\leq& \frac{\sqrt n}{2}A_1^{2}\|u\|^2.
\end{split}
\end{equation}
So,
\begin{equation}\label{eqn-2}
-u^*(D+D^*)u\leq (A_2+\sqrt n A_1^{2})\|u\|^2.
\end{equation}
Furthermore,
\begin{equation}
-\sum_{k.l,\lam,\mu=1}^n\bar u_kR_{\lam\bmu k\bl}\rho_\bla\rho_\mu u_l\leq \frac{1}{2}K\|u\|^2
\end{equation}
and, similarly as in \eqref{eqn-tau-square}
\begin{equation}
\lf|\bar u_k(\tau_{k\nu}^\bmu\tau_{\bla\bl}^\nu+\tau_{\mu k}^{\bar\nu}\tau_{\bl\bar\nu}^\lam)\rho_\lam\rho_\bmu u_l\ri|\leq \sqrt n A_1^2\|u\|^2.
\end{equation}
Hence
\begin{equation}\label{eqn-3}
-u^*Eu\leq  \left(\frac{1}{2}K+\sqrt n A_1^2\right)\|u\|^2.
\end{equation}
Combining \eqref{eqn-ricati},\eqref{eqn-1},\eqref{eqn-2} and \eqref{eqn-3}, we get
\begin{equation}\label{eqn-ricati-X}
\frac{dX}{d\rho}+X^2+AX+XA^*\leq \left[\left(2\sqrt n+\frac{1}{2}\right)A_1^2+A_2+\frac{1}{2}K\right]I_n.
\end{equation}
Moreover, by Lemma \ref{lem-comp-hess}, and that
\begin{equation}
\rho_{,k\bl}\sim\frac{1}{\rho}(\delta_{k\bl}-\rho_k\rho_\bl)
\end{equation}
as $\rho\to 0^+$(See for example \cite{TY2}), we have
\begin{equation}\label{eqn-initial-X}
X\leq \left(\frac{1}{\rho}+\frac{\sqrt 2}{2} A_1\right)I_n
\end{equation}
as $\rho\to 0^+$.

Let $Y=\left[\frac{1}{\rho}+\left((4\sqrt n+3)A_1^{2}+2A_2+K\right)^\frac{1}{2}\right]I_n$, noting that
$$(A+A^*)\geq -\sqrt 2 A_1I_n$$
 and \eqref{eqn-ricati-X}, we have
\begin{equation}
\begin{split}
&\frac{dY}{d\rho}+Y^2+AY+YA^*\\
=&\left\{-\frac{1}{\rho^2}+\left[\frac{1}{\rho}+\left((4\sqrt n+3)A_1^{2}+2A_2+K\right)^\frac{1}{2}\right]^2\right\}I_n\\
&+\left[\frac{1}{\rho}+\left((4\sqrt n+3)A_1^{2}+2A_2+K\right)^\frac{1}{2}\right](A+A^*)\\
\geq&\lf[\left((4\sqrt n+3)A_1^{2}+2A_2+K\right)-\sqrt 2A_1\left((4\sqrt n+3)A_1^{2}+2A_2+K\right)^\frac{1}{2}\ri]I_n\\
\geq&\lf[\left((4\sqrt n+3)A_1^{2}+2A_2+K\right)-\left((4\sqrt n+3)A_1^{2}+2A_2+K\right)/2-A_1^2\ri]I_n\\
= &\left[\left(2\sqrt n+\frac{1}{2}\right)A_1^2+A_2+\frac{1}{2}K\right]I_n\\
\geq&\frac{dX}{d\rho}+X^2+AX+XA^*\\
\end{split}
\end{equation}
where we have also used that $\sqrt 2xy\leq x^2+y^2/2$.

Moreover,
\begin{equation}
Y\geq X
\end{equation}
as $\rho\to 0^+$ by \eqref{eqn-initial-X}. By comparison of matrix Ricatti equations \cite{R}, we have
\begin{equation}
X(\rho)\leq Y(\rho)
\end{equation}
for all $\rho$ with in the cut-locus of $o$. This completes the proof of the theorem.
\end{proof}

In the following, we give a diameter estimate for almost Hermitian manifolds. We first extend the notion of quasi-holomorphic sectional curvature in \cite{CY} for Hermitian
manifolds to almost Hermitian manifolds.
\begin{defn}
Let $(M,J,g)$ be an almost Hermitian manifold. Let $X$ be a real unit vector on $M$.
Define the quasi holomorphic sectional curvature $QH(X)$ as
\begin{equation}
QH(X)=R_{1\bar 1 1\bar 1}-\sum_{i=2}^n|\tau_{i1}^1+\tau_{i 1}^{\bar 1}|^2
\end{equation}
where we have fixed a unitary frame $(e_1,e_2,\cdots,e_n)$ with $e_1=\frac{1}{\sqrt 2}(X-\ii JX)$.
\end{defn}
\begin{rem}
When the complex structure is integrable, the definition of quasi holomorphic sectional curvature
is the same as that in \cite{CY}.
\end{rem}

\begin{thm}
Let $(M,J,g)$ be a complete almost Hermitian manifold and the quasi holomorphic sectional
curvature is not less than $K>0$. Then $d(M)\leq \pi/\sqrt K$.
\end{thm}
\begin{proof} Fixed a unitary frame $(e_1,e_2,\cdots,e_n)$, using Lemma \ref{lem-evolv-1-order} and
the eighth equality of \eqref{eqn-2-order} inLemma \ref{lem-evolv-2-order}, noting that $\tau$ and $R$ are both skew symmetric, we have
\begin{equation}\label{eqn-hess}
\begin{split}
&\frac{d}{d\rho}(\rho_{k\bl}\rho_\bk\rho_l)\\
=&(\rho_{k\bl}\rho_\bk\rho_l)_i\rho_\bi+(\rho_{k\bl}\rho_\bk\rho_l)_\bi\rho_i\\
=&(\rho_{k\bl i}\rho_\bi+\rho_{k\bl\bi}\rho_i)\rho_\bk\rho_l+\rho_{k\bl}(\rho_{\bk i}\rho_\bi+\rho_{\bk \bi}\rho_i)\rho_l+\rho_{k\bl}\rho_\bk(\rho_{li}\rho_\bi+\rho_{l\bi}\rho_i)\\
=&-(\rho_{i\bl}\rho_{\bi k}+\rho_\bi\tau_{ik}^\lambda\rho_{\lambda\bl}+\rho_{k\bla}\tau_{\bi\bl}^{\bla}\rho_i+\rho_{ik}\rho_{\bi\bl}+\rho_{k\lambda}\tau_{\bi\bl}^\lambda\rho_i+\rho_\bi\tau_{ik}^\bla\rho_{\bl\bla}\\
&+R_{\bl\lambda ik}\rho_\bla\rho_\bi+R_{k\bla\bi\bl}\rho_\lambda\rho_i+R_{i\bla k\bl}\rho_\lambda\rho_\bi)\rho_\bk\rho_l+\rho_{k\bl}(\rho_{i \bk}\rho_\bi+\rho_{\bi\bk }\rho_i+\tau_{\bk\bi}^\lam\rho_\lam\rho_i+\tau_{\bk\bi}^\bla\rho_\bla\rho_i)\rho_l\\
&+\rho_{k\bl}\rho_\bk(\rho_{il}\rho_\bi+\tau_{li}^\lam\rho_\lam\rho_\bi+\tau_{li}^\bla\rho_\bla\rho_\bi+\rho_{\bi l}\rho_i)\\
=&-(\rho_{i\bl}\rho_{\bi k}\rho_l\rho_\bk+\rho_{ik}\rho_{\bi\bl}\rho_\bk\rho_l)-R_{i\bj k\bl}\rho_\bi\rho_j\rho_\bk\rho_l+\rho_{k\bl}(\tau_{\bk\bi}^\lam\rho_\lam+\tau_{\bk\bi}^\bla\rho_\bla)\rho_i\rho_l+\rho_{k\bl}\rho_\bk(\tau_{li}^\lam\rho_\lam+\tau_{li}^\bla\rho_\bla)\rho_\bi\\
=&-[\rho_{i\bl}\rho_{\bi k}\rho_l\rho_\bk+(-\rho_{k}\rho_{i\bk}+\tau_{ik}^\lam\rho_\lam\rho_\bk+\tau_{ik}^\bla\rho_\bla\rho_\bk)(-\rho_{\bl}\rho_{l\bi}+\tau_{\bi\bl}^\lam\rho_\lam\rho_l+\tau_{\bi\bl}^\bla\rho_\bla\rho_l)]\\
&-R_{i\bj k\bl}\rho_\bi\rho_j\rho_\bk\rho_l+\rho_{k\bl}(\tau_{\bk\bi}^\lam\rho_\lam+\tau_{\bk\bi}^\bla\rho_\bla)\rho_i\rho_l+\rho_{k\bl}\rho_\bk(\tau_{li}^\lam\rho_\lam+\tau_{li}^\bla\rho_\bla)\rho_\bi\\
=&-2\rho_{k\bi}\rho_{i\bl}\rho_l\rho_\bk-(\tau_{ik}^\lam\rho_\lam\rho_\bk+\tau_{ik}^\bla\rho_\bla\rho_\bk)(\tau_{\bi\bl}^\lam\rho_\lam\rho_l+\tau_{\bi\bl}^\bla\rho_\bla\rho_l)\\
&-R_{i\bj k\bl}\rho_\bi\rho_j\rho_\bk\rho_l+2\rho_{k\bl}(\tau_{\bk\bi}^\lam\rho_\lam+\tau_{\bk\bi}^\bla\rho_\bla)\rho_i\rho_l+2\rho_{k\bl}\rho_\bk(\tau_{li}^\lam\rho_\lam+\tau_{li}^\bla\rho_\bla)\rho_\bi.\\
\end{split}
\end{equation}
Assume that $e_1=\frac{1}{\sqrt 2}(\nabla \rho-\ii J\nabla \rho)$. Then,
\begin{equation}
\rho_1=\rho_{\bar 1}=\frac{1}{\sqrt 2}\ \mbox{and}\ \rho_\a=\rho_\ba=0\ \mbox{for $\alpha>1$.}
\end{equation}
Then, let $f=\rho_{k\bl}\rho_\bk\rho_l=\rho_{1\bar 1}/2$, by (\ref{eqn-hess}), we know that
\begin{equation}
\begin{split}
&\frac{df}{d\rho}\\
=&-4f^2-\frac{1}{4}R_{1\bar 1 1\bar 1}\\
&-\sum_{i=2}^n(|\rho_{1\bi}|^2-2\mbox{Re}\{\rho_{1\bi}(\tau_{i1}^{\bar 1}+\tau_{i1}^1)/\sqrt 2\}+\frac{1}{4}|\tau_{i1}^1+\tau_{i1}^{\bar 1}|^2)\\
\leq&-4f^2-\frac{1}{4}(R_{1\bar 1 1\bar 1}-\sum_{i=2}^n|\tau_{i1}^1+\tau_{i1}^{\bar 1}|^2)\\
\leq &-4f^2-\frac{K}{4}
\end{split}
\end{equation}
Moreover, by Lemma \ref{lem-comp-hess}, we have
\begin{equation}
\rho_{k\bl}\rho_\bk\rho_l=\rho_{,k\bl}\rho_\bk\rho_l+\frac{1}{2}(\tau_{k\lam}^l\rho_\bla+\tau_{\bl\bla}^\bk \rho_\lam)\rho_\bk\rho_l=\rho_{,k\bl}\rho_\bk\rho_l\sim \frac{1}{4\rho}
\end{equation}
as $\rho\to 0$. By comparison of Riccati equation \cite{R}, we know that
\begin{equation}
f\leq \frac{\sqrt{K}}{4}\cot(\sqrt K\rho).
\end{equation}
Hence, by a classical argument(See for example \cite{Li}), we get the conclusion.
\end{proof}
\begin{rem}
The diameter estimate above was disguised with a seemingly different curvature assumption in \cite{Gray1}. Indeed, using the curvature identities in \cite{Yu3}, one can find that the two curvature assumptions in \cite{Gray1} and in the above are the same.
\end{rem}
\section{First eigenvalue estimate for quasi K\"ahler manifolds}
In this section, we give a sharp first eigenvalue estimate for quasi K\"ahler manifolds. We first recall the definition and some properties of quasi K\"ahler manifolds.

Let $(M,J)$ be an almost complex manifold. Since $J$ is not necessary integrable, the exterior differentiation $d\alpha$ of a $(p,q)$-form $\alpha$ on an almost complex manifold $(M,J)$ has four components: $(p-1,q+2)$ component, $(p,q+1)$ component, $(p+1,q)$ component and $(p+2,q-1)$ component. We denote the $(p-1,q+2)$ component of $d\alpha$ as $\bar A\alpha$, the $(p,q+1)$ component as $\dbar\alpha$, the $(p+1,q)$ component as $\partial \alpha$ and  the $(p+2,q-1)$ component as $A\alpha$.
\begin{defn}
An almost Hermitian manifold $(M,J,g)$ is called a quasi K\"ahler manifold if $\dbar\omega_g=0$.
\end{defn}

The following criterion for quasi K\"ahlerity is well known.
\begin{prop}[\cite{k,twy}]\label{prop-crt-quasi-kahler}
Let $(M,J,g)$ be an almost Hermitian manifold. Then, it is quasi K\"ahler if and only if $\tau_{ij}^k=0$ for any $i,j$ and $k$.
\end{prop}

Applying Proposition \ref{prop-crt-quasi-kahler} to Lemma \ref{lem-comp-hess},Lemma \ref{lem-comp-laplace} and Lemma \ref{lem-comp-div}, we have the following corollary.
\begin{cor}\label{cor-comp-quasi-kahler}
Let $(M,g,J)$ be a quasi K\"ahler manifold. Then $f_{i\bar j}=f_{,i\bj}$, $\Delta f=\Delta^L f$ and $\mbox{div} X=\mbox{div}_LX$.
\end{cor}

Because $\Delta$ coincides with
$\Delta^L$, by the same technique as in \cite{F}, we have the following estimate of the first eigenvalue which generalizes the eigenvalue estimate on compact K\"ahler manifolds.
Before stating the first eigenvalue estimate, we need the following definition of quasi Ricci curvature.
\begin{defn}
Let $(M,J,g)$ be a quasi K\"ahler manifold and let
\begin{equation}
\mathcal{R}_{i\bj}=R_{i\bj\lam\bla}-\frac{1}{2}(\tau_{\bla\bmu}^j\tau_{\lam i}^\bmu+\tau_{\lam\mu}^\bi\tau_{\bla\bj}^\mu)-\frac{1}{4}\tau_{\lam\mu}^\bi\tau_{\bla\bmu}^j.
\end{equation}
We call $\mathcal{R}_{i\bj}$ the quasi Ricci curvature of the quasi K\"ahler manifold.
\end{defn}

\begin{thm}
Let $(M,J,g)$ be a compact quasi K\"ahler manifold with the quasi Ricci curvature bounded
from below by a positive constant $K$. Then $\lam_1\geq 2K$, where $\lam_1$ is the first
eigenvalue of $(M,g)$.
\end{thm}
\begin{proof}
Let $f$ be an eigenfunction of $-\Delta$ with eigenvalue $\lam_1$. That is
\begin{equation}
\Delta f=-\lam_1f.
\end{equation}
Then, fixed a unitary frame $(e_1,e_2,\cdots,e_n)$, using the Lemma \ref{lem-ricci-identity}, Corollary \ref{cor-comp-quasi-kahler}, Proposition \ref{prop-crt-quasi-kahler} and Proposition \ref{prop-first-bian}, we know that
\begin{equation}
\begin{split}
&-\lambda_1\int_M\|\nabla f\|^2\\
=&-2\lambda_1\int_Mf_if_\bi\\
=&2\int_Mf_{j\bj i}f_\bi+2\int_Mf_{\bj j\bi}f_i\\
=&2\int_M(f_{ji\bj}-R_{j\bla i\bj}f_\lam)f_\bi+2\int_M(f_{\bj\bi j}-R_{\bj\lam \bi j}f_\bla)f_i\\
=&2\int_M (f_{ji}f_\bi)_\bj+2\int_M(f_{\bj\bi}f_i)_\bj-2\int_M f_{ji}f_{\bi\bj}-2\int_M f_{\bj\bi}f_{ij}-2\int_M(R_{j\bla i\bj}f_\lam f_\bi+R_{\bj\lam \bi j}f_\bla f_i)\\
=&-2\int_M (f_{ij}+\tau_{ji}^\bla f_\bla)f_{\bi\bj}-2\int_M (f_{\bi\bj}+\tau_{\bj\bi}^\lam f_\lam)f_{ij}\\
&-2\int_M[(R_{i\bla j\bj}-\tau_{\bj\bmu}^\lam\tau_{ji}^\bmu)f_\lam f_\bi+(R_{\lam \bi j\bj}-\tau_{j\mu}^\bla\tau_{\bj\bi}^\mu)f_\bla f_i]\\
=&-4\int_Mf_{ij}f_{\bi\bj}+4\int_M\mbox{Re}\{f_{ij}\tau_{\bi\bj}^\lam f_\lam\}-4\int_M\left[R_{i\bj\lam\bla}-\frac{1}{2}(\tau_{\bla\bmu}^j\tau_{\lam i}^\bmu+\tau_{\lam\mu}^\bi\tau_{\bla\bj}^\mu)\right]f_\bi f_j\\
=&-4\int_M\sum_{i,j=1}^n\left|f_{ij}-\frac{1}{2}\tau_{ij}^\bla f_\bla\right|^2-4\int_M\left[R_{i\bj\lam\bla}-\frac{1}{2}(\tau_{\bla\bmu}^j\tau_{\lam i}^\bmu+\tau_{\lam\mu}^\bi\tau_{\bla\bj}^\mu)-\frac{1}{4}\tau_{\lam\mu}^\bi\tau_{\bla\bmu}^j\right]f_\bi f_j\\
\leq&-4\int_M\mathcal{R}_{i\bj}f_\bi f_j\\
\leq&-2K\int_M\|\nabla f\|^2.
\end{split}
\end{equation}
Hence
\begin{equation}
\lam_1\geq 2K.
\end{equation}
\end{proof}

For the equality case, we show that the equality can also be achieved by non-K\"ahler manifolds. Let $\mathbb{S}^6$ equipped with the standard almost complex structure and standard Riemannian metric. Then, $\mathbb{S}^6$ becomes a nearly K\"ahler manifold. For the this nearly K\"ahler manifold, $R^L_{i\bj}=5\delta_{ij}$, and by \cite{Yu3}, $R_{i\bj}=0$. By the curvature identity
$$R^L_{i\bj}=R_{i\bj}+\frac{5}{4}\sum_{\lambda,\mu=1}^n\tau_{\lambda\mu}^\bi\tau_{\bla\bmu}^j$$
in \cite{Yu3}, we have
\begin{equation}
\sum_{\lambda,\mu=1}^n\tau_{\lambda\mu}^\bi\tau_{\bla\bmu}^j=4\delta_{ij}.
\end{equation}
Therefore, the quasi Ricci curvature
\begin{equation}
\mathcal{R}_{i\bj}=R_{i\bj}+\frac{3}{4}\sum_{\lambda,\mu=1}^n\tau_{\lambda\mu}^\bi\tau_{\bla\bmu}^j=3\delta_{ij}.
\end{equation}
where we have used Lemma \ref{lem-criterion-nearly-kahler} in the next section. So, the constant $K$ in the last theorem is $3$. It is clear that the first eigenvalue of the standard metric on $\mathbb{S}^6$ is $6$. Hence, equality of the last theorem is achieved by the nearly K\"ahler manifold $\mathbb{S}^6$.
\section{Sharp Hessian comparison on nearly K\"ahler manifolds}
In this section, by using the Bochner techniques in \cite{LW}, we obtain a sharp Hessian comparison on nearly K\"ahler manifolds generalizing the results of \cite{LW,TY2}.

Recall the definition of nearly K\"ahler manifolds.
\begin{defn}
Let $(M,J,g)$ be an almost Hermitian manifold. It is called nearly K\"ahler if $(D_XJ)X=0$ for any tangent vector $X$.
\end{defn}

The following lemma is well know, see for example \cite{Gray3}.
\begin{lem}
Let $(M,J,g)$ be a quasi K\"ahler manifold, then
\begin{equation}
\nabla_XY=D_XY-\frac{1}{2}J(D_XJ)(Y).
\end{equation}
for any tangent vector fields $X$ and $Y$.
\end{lem}
The following corollary is straight forward by the definition of nearly K\"ahler manifolds and the last lemma.
\begin{cor}
Let $(M,J,g)$ be a nearly K\"ahler manifold. Then $\na_XX=D_XX$
for any tangent vector field $X$.
\end{cor}
The following criterion for nearly K\"ahler manifold is well known, see for example \cite{Na1,Na2}.
\begin{lem}\label{lem-criterion-nearly-kahler}
 An almost Hermitian manifold $(M,J,g)$ is nearly K\"ahler if and only if $\tau_{ij}^k=0$ and $\tau_{ij}^\bk=\tau_{jk}^\bi$ for all $i,j$ and $k$ when we fix a (1,0)-frame.
\end{lem}

Moreover, it turns our that the torsion is parallel for nearly K\"ahler manifolds.
\begin{thm}[Kirichenko \cite{Ki,Yu3}]\label{thm-para-tor}
Let $(M,J,g)$ be nearly K\"ahler manifold. Then $\nabla\tau=0$.
\end{thm}
Applying Lemma \ref{lem-criterion-nearly-kahler} and Theorem \ref{thm-para-tor} to
Proposition \ref{prop-first-bian}, we have the following first Bianchi identities
for nearly K\"ahler manifolds.
\begin{cor}\label{cor-first-bian}
Let $(M,J,g)$ be a nearly K\"ahler manifold and fixed a unitary frame. Then
\begin{enumerate}
\item $R_{i\bj kl}=0$;
\item $R_{i\bj k\bl}-R_{k\bj i\bl}=-\tau_{ik}^\bla\tau_{\bj\bl}^\lam$;
\item $R_{i\bj k\bl}-R_{i\bl k\bj}=-\tau_{ik}^\bla\tau_{\bj\bl}^\lam$;
\item $R_{i\bj k\bl}=R_{k\bl i\bj}$.
\end{enumerate}
\end{cor}
By (4) of the above corollary, the first Ricci curvature and second Ricci curvature for nearly K\"ahler manifolds coincides, so we simply denote them as $R_{i\bj}$.

\begin{lem}\label{lem-pre-nearly-kahler}
Let $(M,J,g)$ be a nearly K\"ahler manifold, $o$ be a fixed point and $\rho(x)$ the the distance from
$x$ to $o$. Let $\gamma$ be a normal geodesic starting from $o$. Let $(e_1,e_2,\cdots,e_n)$ be a
unitary frame parallel along $\gamma$ with respect to the canonical connection with $e_1=\frac{1}{\sqrt 2}(\gamma'(0)-J\gamma'(0))$. Then $e_1=\frac{1}{\sqrt 2}(\gamma'-\ii J\gamma')$ all over
the $\gamma$, $\rho_{1}=\rho_{\bar 1}=\frac{1}{\sqrt 2}$, $\rho_\a=\rho_{\bar \a}=0$  for all $1<\alpha$ and $\rho_{i1}=-\rho_{i\bar1}$ for all $i\geq 1$.
with in the cut-locus of $o$.

\end{lem}
\begin{proof}
Note that $\nabla_{\gamma'}\gamma'=D_{\gamma'}\gamma'=0$, we know that $e_1=\frac{1}{\sqrt 2}(\gamma'-J\gamma')$ all over
$\gamma$. It is clear that $e_1$ is also parallel along $\gamma$ with respect to the Leiv-Civita connection.
Moreover $e_1=\frac{1}{\sqrt 2}(\nabla\rho-J\nabla\rho)$. Hence
\begin{equation}
\rho_1=\vv<\nabla\rho,e_1>=\frac{1}{\sqrt 2},
\end{equation}
and
\begin{equation}
\rho_{\a}=\vv<\nabla\rho,e_\a>=0
\end{equation}
for all $\a>1$. By these and Lemma \ref{lem-evolv-1-order}, we know that
\begin{equation}
\rho_{i1}=-\rho_{i\bar 1}
\end{equation}
for all $i\geq 1$.
\end{proof}
\begin{defn}
On a nearly K\"ahler manifold, define
\begin{equation}
\mathcal{R}(X,\bar X,Y,\bar Y)=R(X,\bar X,Y,\bar Y)+\|\tau(X,Y)\|^2
\end{equation}
for any (1,0) vectors $X$ and $Y$.
\end{defn}
\begin{defn}
Let $(M,J,g)$ be a nearly K\"ahler manifold, we say that its quasi holomorphic bisectional curvature$\geq K$
if
\begin{equation}
\frac{\mathcal{R}(X,\bar X,Y,\bar Y)}{\|X\|^2\|Y\|^2+|\vv<X,\bar Y>|^2}\geq K
\end{equation}
for any two nonzero $(1,0)$ vectors $X$ and $Y$.
\end{defn}
\begin{thm}
Let $(M,J,g)$ be a complete nearly K\"ahler manifold and $o$ be a fixed point in $M$. Let $B_o(R)$ be a geodesic ball within the cut-locus of $p$. Suppose that the quasi holomorphic bisectional curvature on $B_o(R)$ is not less than $K$ where $K$ is a constant. Then
\begin{equation}
 \rho_{\alpha\bar\beta}\leq\left\{\begin{array}{ll}
 \sqrt{K/2}\cot(\sqrt{K/2}\rho)(g_{\alpha\bar\beta}-2\rho_\alpha \rho_{\bar\beta})+\sqrt{2K}\cot(\sqrt{2K}\rho)\rho_\alpha \rho_{\bar\beta}&(K>0)\\
 \frac{1}{r}(g_{\alpha\bar\beta}-\rho_\alpha \rho_{\bar\beta})&(K=0)\\
 \sqrt{-K/2}\coth(\sqrt{-K/2}\rho)(g_{\alpha\bar\beta}-2\rho_\alpha \rho_{\bar\beta})+\sqrt{-2K}\coth(\sqrt{-2K}\rho)\rho_\alpha \rho_{\bar\beta}&(K<0)\\
 \end{array}\right.
 \end{equation}
 in $B_o(R)$ with equality holds all over $B_o(R)$ if and only if $B_o(R)$ is holomorphic and isometric equivalent to the geodesic ball with radius $R$ in the K\"ahler space form of constant holomorphic bisectional curvature $K$, where $\rho$ is the distance function to the fixed point $o$.
\end{thm}
\begin{proof}
Let $\gamma$ be a geodesic starting from $o$, and $(e_1,e_2,\cdots,e_n)$ be the same as
in the last lemma. Then, by Lemma \ref{lem-criterion-nearly-kahler}
Lemma \ref{lem-evolv-2-order}, Corollary \ref{cor-first-bian} and Lemma \ref{lem-pre-nearly-kahler}, we know that
\begin{equation}
\begin{split}
&\rho_{k\bl i}\rho_\bi+\rho_{k\bl \bi}\rho_i\\
=&-\rho_{i\bl}\rho_{\bi k}-\rho_{ik}\rho_{\bi\bl}-\rho_{\lam k}\tau_{\bi\bl}^\lambda\rho_i-\rho_\bi\tau_{ik}^\bla\rho_{\bla\bl}-(R_{i\bla k\bl}+\tau_{k\mu}^\bi\tau_{\bla\bl}^\mu+\tau_{ik}^\bmu\tau_{\bl\bmu}^\lam)\rho_\lambda\rho_\bi\\
=&-\rho_{i\bl}\rho_{\bi k}-\rho_{1k}\rho_{\bar1\bl}-\sum_{\a=2}^n\rho_{\a k}\rho_{\ba\bl}+\frac{1}{\sqrt 2}(\rho_{\lam k}\tau_{\bla\bl}^1+\tau_{\lam k}^{\bar 1}\rho_{\bla\bl})-\frac{1}{2}(R_{1\bar 1k\bar l}+2\tau_{\lam k}^{\bar 1}\tau_{\bla\bl}^{1})\\
=&-\rho_{k\bi}\rho_{i\bl}-\rho_{k\bar 1}\rho_{1\bl}-\sum_{\a=2}^n\rho_{\a k}\rho_{\ba\bl}+\frac{1}{\sqrt 2}\lf(\sum_{\a=2}^n\rho_{\a k}\tau_{\ba\bl}^1+\sum_{\a=2}^n\tau_{\a k}^{\bar 1}\rho_{\ba\bl}\ri)-\frac{1}{2}\lf(R_{1\bar 1k\bar l}+2\sum_{\alpha=2}^n\tau_{\a k}^{\bar 1}\tau_{\ba\bl}^{1}\ri).\\
\end{split}
\end{equation}
Let $X=(\rho_{k\bar l})^{l=1,2,\cdots,n}_{k=1,2,\cdots,n}$, $B=(\rho_{\bk\bl})_{k=2,3,\cdots,n}^{l=1,2,\cdots,n}$, $C=\frac{1}{\sqrt 2}(\tau_{\bk\bl}^1)_{k=2,3,\cdots,n}^{l=1,2,\cdots,n}$, $D=\lf(-\frac{1}{2}(R_{1\bar 1k\bar l}+2\sum_{\alpha=2}^n\tau_{\a k}^{\bar 1}\tau_{\ba\bl}^{1})\ri)_{k=1,2,\cdots,n}^{l=1,2,\cdots,n}$ and $X_1$ be the first
column of $X$. Then
\begin{equation}
\begin{split}
&\frac{dX}{d\rho}+X^2+X_1X_1^{*}\\
=&-B^*B+B^*C+C^*B+D\\
=&-(B-C)^*(B-C)+C^*C+D\\
\leq&C^*C+D\\
\leq&\left(\begin{array}{cc}-K&0\\0&-\frac{K}{2}I_{n-1}
\end{array}\right).
\end{split}
\end{equation}
Then, by the same argument as in \cite{TY2}, we have
\begin{equation}
 X\leq\left\{\begin{array}{ll}\left(\begin{array}{cc}\frac{\sqrt{2K}}{2}\cot(\sqrt {2K} \rho)&0\\
0&\sqrt{K/2}\cot(\sqrt {K/2} \rho)I_{n-1}\\
\end{array}\right)&(K>0)\\
\left(\begin{array}{cc}\frac{1}{2\rho}&0\\
0&\frac{1}{\rho}I_{n-1}\\
\end{array}\right)&(K=0)\\
\left(\begin{array}{cc}\frac{\sqrt{-2K}}{2}\coth(\sqrt {-2K} \rho)&0\\
0&\sqrt{-K/2}\coth(\sqrt {-K/2} \rho)I_{n-1}\\
\end{array}\right)&(K<0).\\
\end{array}\right.
\end{equation}
This is the equality in the conclusion of the theorem.

If the equality holds, we have $\rho_{kl}=\frac{1}{\sqrt 2}\tau_{kl}^{\bar 1}$ for $k,l=2,3,\cdots, n$. By Lemma \ref{lem-ricci-identity}, we have
\begin{equation}
\rho_{kl}=\rho_{lk}+\tau_{kl}^\bla\rho_\bla=\frac{1}{\sqrt 2}\tau_{lk}^{\bar 1}+\frac{1}{\sqrt 2}\tau_{kl}^{\bar 1}=0
\end{equation}
for all $k,l=2,3,\cdots,n$. Hence
\begin{equation}
\tau_{kl}^{\bar 1}=0
\end{equation}
for all $k,l=1,2,\cdots,n$. In particular, at the point $o$, we have
\begin{equation}
\tau_{ij}^\bk(o)=0
\end{equation}
for all $i,j$ and $k$. By Theorem \ref{thm-para-tor}, we know that $\tau=0$ and hence $M$ is K\"ahler. At this position, the same argument in \cite{TY2}, we obtain the conclusion when equality holds.
\end{proof}

By the Hessian comparison, we have the following direct corollaries.
\begin{cor}
Let $(M,J,g)$ be a complete nearly K\"ahler manifold and $o$ be a fixed point in $M$. Let $B_o(R)$ be a geodesic ball within the cut-locus of $p$. Suppose that the quasi holomorphic bisectional curvature on $B_o(R)$ is not less than $K$ where $K$ is a constant. Then
\begin{equation}
 \Delta \rho\leq\left\{\begin{array}{ll}
 \sqrt{2K}\cs{\cot\cs{\sqrt{2K}\rho}+(n-1)\cot\cs{\sqrt{K/2}\rho}}&(K>0)\\
 \frac{2n-1}{\rho}&(K=0)\\
 \sqrt{-2K}\cs{\coth\cs{\sqrt{-2K}\rho}+(n-1)\coth\cs{\sqrt{-K/2}\rho}}&(K<0)\\
 \end{array}\right.
 \end{equation}
 in $B_o(R)$ with equality holds all over $B_o(R)$ if and only if $B_o(R)$ is holomorphic and isometric equivalent to the geodesic ball with radius $R$ in the K\"ahler space form of constant holomorphic bisectional curvature $K$, where $\rho$ is the distance function to the fixed point $o$.
\end{cor}
By the same argument as in \cite{Ch} (See also \cite{Li}), we have the following comparison of eigenvalues for nearly K\"ahler manifolds.
\begin{cor}
Let $(M,J,g)$ be a complete nearly K\"ahler manifold and $o$ be a fixed point in $M$. Let $B_o(R)$ be a geodesic ball within the cut-locus of $p$. Suppose that the quasi holomorphic bisectional curvature on $B_o(R)$ is not less than $K$ where $K$ is a constant. Then
\begin{equation}
\lambda_1(B_o(R))\leq \lambda_1(B_K(R))
\end{equation}
where $B_K(R)$ is the geodesic ball with radius $R$ in the K\"ahler space form with constant holomorphic bisectional curvature $K$. Moreover, if the equality holds, then $B_o(R)$ and $B_K(R)$ are holomorphically isometric to each other.
\end{cor}
\begin{cor}
Let $(M,J,g)$ be a complete nearly K\"ahler manifold and $o$ be a fixed point in $M$. Let $B_o(R)$ be a geodesic ball within the cut-locus of $p$. Suppose that the quasi holomorphic bisectional curvature on $B_o(R)$ is not less than $K$ where $K$ is a constant. Then
\begin{equation}
V_o(R)\leq V_K(R)
\end{equation}
where $V_K(R)$ is the volume of $B_K(R)$. Moreover, if the equality holds, then $B_o(R)$ and $B_K(R)$ are holomorphically isometric to each other.
\end{cor}
\begin{cor}
Let $(M,J,g)$ be a complete nearly K\"ahler manifold with quasi holomorphic bisectional curvature $\geq K$ with $K>0$. Then
\begin{equation}
V(M)\leq V(\mathbb{CP}^n_K)
\end{equation}
where $\mathbb{CP}^n_K$ means $\mathbb{CP}^n$ equipped with a K\"ahler metric with constant bisectional curvature $K$. Moreover, if the equality holds, $M$ is holomorphically isometric to $\mathbb{CP}^n_K$.
\end{cor}

\end{document}